\documentclass[11pt,tbtags]{amsart}

\usepackage{amsmath,amssymb,mathtools,amsthm,bm}
\usepackage{xcolor,graphicx,float,enumerate}
\usepackage{latexsym,enumerate}
\usepackage{amsmath,amsthm,amsopn,amstext,amscd,amsfonts,amssymb}
\usepackage{color}

\usepackage[normalem]{ulem}

\newcommand{\rn}{\mathbb{R}^n}

\newcommand{\dint}{\displaystyle\int}
\newcommand{\R}{\mathbb{R}}
\newcommand{\N}{\mathbb{N}}

\DeclareMathOperator{\diver}{div}

\usepackage[utf8]{inputenc}
\DeclareUnicodeCharacter{00A0}{ }

\usepackage{geometry}
\newtheorem{proposition}{Proposition}[section]
\newtheorem{theorem}{Theorem}[section]
\newtheorem{corollary}{Corollary}[section]
\newtheorem{lemma}{Lemma}[section]
\theoremstyle{definition}
\newtheorem{definition}{Definition}[section]
\newtheorem{remark}{Remark}[section]

\usepackage[hidelinks]{hyperref}
\usepackage[nameinlink,noabbrev]{cleveref}

\numberwithin{equation}{section}

\begin{document}
\title{Isoperimetric sets for weighted twisted eigenvalues}

\author{B. Brandolini
	\and A. Henrot
	 \and 
	A. Mercaldo
	\and 
	M.R. Posteraro
		}
		
\begin{abstract} 
In tis paper we prove an isoperimetric inequality for the first twisted eigenvalue $\lambda_{1,\gamma}^T(\Omega)$ of a weighted operator, defined as the minimum of the usual Rayleigh quotient when the trial functions belong to the weighted Sobolev space $H_0^1(\Omega,d\gamma)$ and have weighted mean value equal to zero in $\Omega$. We are interested in positive measures $d\gamma=\gamma(x) dx$ for which we are able to identify the isoperimetric
sets, namely, the sets that minimize $\lambda_{1,\gamma}^T(\Omega)$ among sets of given weighted measure. In the cases under consideration, the optimal sets are given by two identical and disjoint copies of the isoperimetric sets (for the weighted perimeter with respect to the weighted measure).
\end{abstract}

\maketitle

%%%%%%%%%%%%%%%%%%%%%%%%%%%%

\section{Introduction}

Assume that $\Omega$ is a bounded, open subset of $\rn$. It is well-known that, among sets of given measure, the first eigenvalue $\lambda_1^D(\Omega)$ of the classical Dirichlet-Laplacian is minimized by balls. Namely, the celebrated Faber-Krahn inequality states that
$$
\lambda_1^D(\Omega)|\Omega|^{2/n} \ge \lambda_1^D(B)|B|^{2/n},
$$ 
equality holding if and only if $\Omega=B$ is a ball (up to sets of zero capacity). As everyone knows, the first eigenvalue $\lambda_1^D(\Omega)$ can be variationally characterized as follows
$$
\lambda_1^D(\Omega)=\min\left\{
\frac{\dint_\Omega |\nabla \phi|^2\, dx}{\dint_\Omega \phi^2\, dx}: \phi \in H_0^1(\Omega)\setminus\{0\}
\right\}.
$$
Hence a natural question arises: is it possible to derive similar lower bounds for the first eigenvalue of the Laplacian defined in subset of $H^1(\Omega)$ other that $H_0^1(\Omega)$? The answer is negative, if we consider the subset of $H^1(\Omega)$ filled by Sobolev functions having zero mean value. Indeed, it is well-known that the first non trivial eigenvalue $\lambda_1^N(\Omega)$ of the Neumann-Laplacian
(here $\Omega$ has to be smooth enough, for example Lipschitz), that is
$$
\lambda_1^N(\Omega)=\min\left\{
\frac{\dint_\Omega |\nabla \phi|^2\, dx}{\dint_\Omega \phi^2\, dx}: \phi \in H^1(\Omega)\setminus\{0\}, \int_\Omega \phi\, dx=0
\right\},
$$ 
cannot be estimated from below in terms of the measure of $\Omega$, and it is actually maximized by balls. More precisely, the famous Szeg\H{o}-Weinberger inequality (\cite{S,W}) holds true
$$
\lambda_1^N(\Omega)|\Omega|^{2/n}\le \lambda_1^N(B)|B|^{2/n},
$$
with equality sign if and only if $\Omega=B$ is a ball. We can hope to gain a universal lower bound depending on the measure by combining the zero mean value condition with another boundary condition. This is done in the definition of twisted eigenvalues of the Laplacian, which have been introduced in \cite{BB}, in the context of constant mean curvature surfaces.
In \cite{BB} the authors consider the following minimization problem
$$
\lambda_1^T(\Omega)=\min\left\{
\frac{\dint_\Omega |\nabla \phi|^2\, dx}{\dint_\Omega \phi^2\, dx}: \phi \in H_0^1(\Omega)\setminus\{0\}, \int_\Omega \phi\, dx=0
\right\}.
$$
or equivalently, the first eigenvalue of the following problem
\begin{equation}\label{babe}
\left\{
\begin{array}{ll}
-\Delta u =\lambda u - \dfrac{1}{|\Omega|}\dint_\Omega \Delta u\, dx \qquad &\mbox{in } \Omega
\\ 
u=0 & \mbox{on }\partial\Omega.
\end{array}
\right.
\end{equation}
Due to the presence of the mean value of the Laplacian, problems like \eqref{babe} are often referred to as nonlocal problems (see also \cite{BDNT,BFNT,BKN,CD,CHP,DGS,DP} and the references therein).

In \cite{FH} (see also \cite{GL})  the problem of minimizing $\lambda_1^T$ among sets with given measure is faced. The authors prove that 
\begin{equation}\label{fh}
\lambda_1^T(\Omega)\ge \lambda_1^T(B_1\cup B_2),
\end{equation}
where $B_1$ and $B_2$ are two disjoint balls having measure $|\Omega|/2$. Moreover, equality holds (for regular $\Omega$) if and only if $\Omega=B_1\cup B_2$. 

\begin{remark}
Since an eigenfunction corresponding to $\lambda_1^T(\Omega)$ changes its sign in $\Omega$, inequality \eqref{fh} has much in common with the Krahn-Szeg\H{o} inequality on the second eigenvalue of the Dirichlet-Laplacian (see, for example, \cite{H}).
\end{remark}

\noindent It is not a case that, in all the inequalities recalled above, balls play a fundamental role. This is because balls are isoperimetric sets for the Lebesgue measure.

In this paper  we are interested in the following nonlocal weighted eigenvalue problem
\begin{equation}\label{prob}
\left\{\begin{array}{ll}
-\diver\left(\gamma(x)\nabla u\right)= \lambda \gamma(x) \,u -\gamma(x)\left(\dfrac{1}{|\Omega|_\gamma}\dint_\Omega \diver\left(\gamma(x)\nabla u\right)\, dx\right)\quad &\mbox{in } \Omega
\\ 
u=0 & \mbox{on } \partial\Omega,
\end{array}
\right.
\end{equation}
where  $|\Omega|_\gamma=\int_\Omega \gamma(x)\, dx$ is the weighted measure of $\Omega$. We are interested in \emph{``good measures''} $d\gamma =\gamma(x) dx$ (see Section 3 for the definition), which allow us to
characterize the first weighted twisted  eigenvalue $\lambda_{1,\gamma}^T(\Omega)$ as follows 
\begin{equation}\label{mineigen}
\lambda_{1,\gamma}^T(\Omega)\, =\inf
 \left \{
 \frac{\displaystyle\int_\Omega|\nabla \phi|^2\, d\gamma}{\displaystyle\int_\Omega \phi^2\, d\gamma} : 
 \phi\in H^1_0(\Omega, d\gamma)\setminus\{0\}, \int_\Omega \phi\, d\gamma=0 \right\},
\end{equation}
and to prove the following
\begin{theorem}\label{teo:main}
Let $d \gamma$ be a ``good measure". Then 
\begin{equation}\label{mainineq} 
\lambda_{1,\gamma}^T(\Omega)\ge \lambda_{1,\gamma}^T(\omega_1\cup \omega_2),
\end{equation}
where $\omega_1$, $\omega_2$ are two disjoint isoperimetric sets (for the weighted perimeter with respect to the measure $d\gamma=\gamma(x)\,dx$),  such that $|\omega_1|_\gamma=|\omega_2|_\gamma=\frac{|\Omega|_\gamma}{2}$.
\end{theorem}
The paper is organized as follows. After introducing some notation in Section 2, the definition of \emph{``good measure"} in Section 3, and discussing some basic properties of twisted weighted eigenvalues and eigenfunctions in Section 4, in Section 5 we reduce the problem of minimizing $\lambda_{1,\gamma}^T$ among sets with given weighted measure to the easier problem of minimizing $\lambda_{1,\gamma}^T$ among unions of two disjoint isoperimetric sets with given weighted measure. Sections 6 and 7 are devoted to completely prove Theorem \ref{teo:main} in the special cases $d\gamma= \pi^{-n/2}e^{-|x|^2}dx$ (Gaussian measure) and $d\gamma=x_n^kdx$ (power like measure).

%%%%%%%%%%%%%%%%%%%%%%%%%%%%%

\section{Notation and preliminary results}

Let $\gamma$ be a positive function on $\rn$ and let us consider the measure $d\gamma=\gamma(x)\,dx$. If  $\Omega$ is a set of locally finite perimeter,  we define the weighted perimeter and the weighted measure of $\Omega$ as follows 
\begin{equation*}
P_\gamma(\Omega)=\int_{\partial \Omega}\gamma(x)\,d\mathcal{H}^{n-1}, \qquad
|\Omega|_\gamma=\int_\Omega \gamma(x)\, dx.
\end{equation*}
We assume that \emph{the weighted isoperimetric problem is solved }for every $m>0$, that is, 
for any fixed positive weighted measure $m > 0$, we assume the existence of a set $
\Omega^\sharp \subset \rn$ of locally finite perimeter such that $
|\Omega^\sharp|_\gamma=m$ and 
\begin{equation}  \label{p1}
P_\gamma(\Omega^\sharp)=\min\left\{P_\gamma(A): \> A \subset \rn, \> A 
\mbox{ of
locally finite perimeter}, \> |A|_\gamma=m\right\}.
\end{equation}
Equivalently, we write
\begin{equation}
P_{\gamma }(\Omega )\geq P_{\gamma }(\Omega ^{\sharp}), \qquad |\Omega^\sharp|_\gamma=|\Omega|_\gamma.
\label{dis_isop}
\end{equation}
We  refer to such a set $\Omega^\sharp$ as an \emph{isoperimetric set with respect to the measure $d\gamma$}.

Now let us define the weighted rearrangement of a measurable function $u\, :\Omega \rightarrow \R$ (see, for example, \cite{CR}). 
If $u$ is such that its distibution function $\mu_u$ satisfies
\begin{equation*}
\mu_u(\theta)=  \left|\left\{ x\in \Omega :|u(x)|>\theta \right\}
\right|_\gamma <\infty ,\qquad \mbox{for every } \theta  > 0,
\end{equation*}
we define the decreasing rearrangement $u^{\ast }$ of $u$ as the generalized inverse of $\mu_u$, that is 
\begin{equation*}
u^{\ast }(s)= \inf \left\{ \theta \ge 0:\>\mu (\theta )\leq
s\right\} ,\qquad s\in \left( 0,|\Omega|_\gamma \right] .
\end{equation*}
We usually say that two functions $u$ and $v$ are equi-distributed, or
equivalently,  we say that $v$ is a rearrangement of $u$, if $u$ and $v$  share the
same distribution function, that is $\mu_u=\mu_v$. 

\noindent The determination of the cases of equality in the
isoperimetric inequality \eqref{dis_isop} leads to the
definition of a particular rearrangement of $u$, the weighted rearrangement $u^{\sharp
}$, as the unique function, which is equi-distributed with $u$ and  whose super-level sets are isoperimetric sets with the same weighted measure as the corresponding super-level sets of $u$.

\noindent Since, by definition, $u$, $u^\ast$ 
and $u^{\sharp }$ are equi-distributed, the Cavalieri principle implies
\begin{equation}
\left\Vert u\right\Vert_{L^{p}(\Omega, d\gamma )} = \Vert u^\ast\Vert_{L^p(0,|\Omega|_\gamma)}=\Vert
u^{\sharp }\Vert_{L^{p}(\Omega^{\sharp }, d\gamma )}, \qquad p \ge 1.
\label{Cavalieri}
\end{equation}

\noindent For our purposes we will also need the following
Hardy-Littlewood inequality.

\begin{lemma}
Let $u,v$  be measurable functions, such that
$$
\mu_u(\theta)<\infty,\qquad  \mu_v(\theta)<\infty, \qquad \mbox{for every } \theta >0.
$$ 
Then
\begin{equation}
\int_{\Omega }\left| u\,v\right| d\gamma\leq \int_{0}^{|\Omega|_\gamma}u^{\ast }(s)v^{\ast }(s)ds =\int_{\Omega^{\sharp } }
u^{\sharp }\,v^{\sharp } d\gamma. \label{HL}
\end{equation}
\end{lemma}

In this paper we consider measures $d\gamma$ such that a  P\'{o}lya-Szeg\H{o} principle with respect to $d\gamma$ holds true. In order to properly state it, we recall the definition of weighted Sobolev space.

 \begin{definition}
 The weighted Sobolev space $H^1(\Omega, d\gamma)$ is the set of all
functions
 $\phi \in W^{1,1}_{loc}(\Omega)$ such that $(\phi,|D\phi|) \in L^2(\Omega,d\gamma)\times
L^2(\Omega,d\gamma)$,
 endowed with the norm
 $$
 ||\phi||_{H^1(\Omega,d\gamma)} =
||\phi||_{L^2(\Omega,d\gamma)}+||D\phi||_{L^2(\Omega,d\gamma)}.
 $$
  The weighted Sobolev space $H_0^1(\Omega,d \gamma)$
  is the closure of $C_0^\infty(\Omega)$ in $H^1(\Omega, d\gamma)$.  
  \end{definition}

Let $\Omega$ be an open subset of $\R^n$ with finite weighted measure and let  $u\in H^1_0(\Omega,d\gamma )$. We say that \emph{a P\'olya-Szeg\H o principle} holds true if  $u^\sharp \in H_0^1(\Omega^\sharp,d\gamma)$ and 
\begin{equation}
\int_{\Omega}|\nabla u|^{2}d\gamma \geq \int_{\Omega^\sharp}|\nabla u^{\sharp}|^{2}d\gamma.  \label{Polya_Szego_Eh}
\end{equation}
 Moreover equality sign holds in \eqref{Polya_Szego_Eh} if and only if $u=u^{\sharp}$ (modulo elementary transformations).

\medskip
\noindent 
We will deal with a class of measures such that \emph{a weighted isoperimetric inequality} and \emph{a P\'olya-Szeg\H o principle} hold trues.  
In the next section we collect some examples of these measure.

%%%%%%%%%%%%%%%%%%%%%%%%%%%%%%%%

\section{Good measures}\label{sec_dens}

In this section we introduce a class of measures, that we will call \emph{``good measures''}, for which Theorem 1.1 holds true, and we provide some examples of measures belonging to such a  class.  The main property (4) below consists in the possibility of separating two isoperimetric sets.

\begin{definition}\label{dens}
Let $\gamma$ be a positive function on $\rn$. We say that  $d\gamma=\gamma (x) dx$ is a \emph{good measure} if 
\begin{enumerate}
\item the isoperimetric problem \eqref{p1} has a solution for any $m>0$;
\item for every open set $\Omega$ with finite weighted measure, the embedding $H_0^1(\Omega,d\gamma) \hookrightarrow L^2(\Omega,d\gamma)$ is compact;
\item for every open set $\Omega$ with finite weighted measure, the P\'olya-Szeg\H o principle \eqref{Polya_Szego_Eh} holds true;
\item for every open set $\Omega$ with finite weighted measure and for any $ m_1,m_2\ge 0$ such that $m_1+m_2= |\Omega|_\gamma$, two isoperimetric sets $\omega_1, \omega_2$ exist such that $|\omega_1|=m_1$, $|\omega_2|=m_2$ and $\omega_1\cap\omega_2=\emptyset$.
\end{enumerate}
\end{definition}

\subsection{Example 1}
Let $d\gamma=dx$, $x \in \rn$, denote the usual Lebesgue measure in $\rn$. Obviously $d\gamma$ is a good measure in $\R^n$. Indeed, the classical isoperimetric inequality holds true with balls as isoperimetric sets, ensuring the validity of conditions (1) and (4); the classical Sobolev embedding theorem holds true ensuring the validity of (2);  if $u^\sharp$ stands for the Schwarz symmetrizand of $u$, the classical P\'olya-Szeg\H o inequality holds true  (see \cite{deGiorgi}, \cite{PS}).

\subsection{Example 2}
Let $d\gamma= \pi^{-n/2} e^{-|x|^2}\, dx$, $x\in \rn$, denote the usual Gaussian measure in $\rn$, normalized so that $\gamma(\rn)=1$. Then $d\gamma)$ is a good measure in $\rn$. Indeed, conditions (1) and (4) are satisfied as a consequence of the isoperimetric inequality  for   Gaussian measure  (see \cite{B,CK,E}), that we can state as follows. 
\begin{theorem}
 Let $\Omega $ be any measurable subset of ${\mathbb{R}}^{n}$
  having finite Gaussian perimeter and let $\Omega ^{\sharp
}$ be the half-space 
\begin{equation*}\label{halfspace}
\Omega^\sharp=\left\{ x=(x_1,x_2,...,x_n) \in\R^n:\>x_1>a\right\} ,  
\end{equation*}
where $a$ is taken so that $|\Omega^\sharp|_\gamma=|\Omega|_\gamma.$ Then
\begin{equation}
P_{\gamma }(\Omega )\geq P_{\gamma }(\Omega ^{\sharp}).
\label{dis_isop_gauss}
\end{equation}
Moreover equality holds in (\ref{dis_isop_gauss}) if and only if $\Omega
=\Omega ^{\sharp},$ modulo a rotation.
\end{theorem}
\begin{remark}
A straightforward calculation gives
\begin{equation*}
a=k^{-1}(|\Omega|_\gamma),
\end{equation*}
where $k(t)$ is the   function
\begin{equation}
k(t)= \frac{1}{\sqrt{\pi }}\int_{t}^{\infty }\exp (-\sigma
^{2})d\sigma\,, \qquad t\in \R.  \label{compl_err_funct}
\end{equation}
\end{remark}

\noindent Moreover, a P\'olya-Sz\'eg\H{o} principle also holds true with respect to the Gaussian measure and therefore condition (3) in Definition \ref{dens} is also satisfied. Indeed, 
according to the definition of weighted rearrangement given in the previous section, 
we can  define the rearrangement of a measurable function $u\, :\Omega \rightarrow \R$, $u^{\sharp
}$, as follows
\[
u^{\sharp }(x)=u^{\ast }(k(x_{1})),\text{ }x\in \Omega ^{\sharp},
\]
where $k$ is defined in (\ref{compl_err_funct}).
By definition, the super-level sets of  $u^{\sharp }$ are half-spaces, so $u^\sharp$ actually depends on one variable only
(say $x_{1}$) and it is an increasing function with respect to it.  The following P\'{o}lya-Sz\'eg\H{o} principle is shown in 
 \cite{CK,E,T}.

\begin{theorem} 
 Let  $u\in H^1_0(\Omega,d\gamma )$. Then $u^\sharp \in H_0^1(\Omega^\sharp,d\gamma)$ and 
\begin{equation*}
\int_{\Omega}|\nabla u|^{2}d\gamma \geq \int_{\Omega^\sharp}|\nabla u^{\sharp}|^{2}d\gamma.  \label{Polya_Szego_Eh2}
\end{equation*}
 Moreover equality holds if and only if $\Omega=\Omega^\sharp$ and $u=u^{\sharp},$ modulo a rotation.
\end{theorem}

\noindent Eventually, condition (2)  in Definition \ref{dens} is also satisfied (see, for example, \cite{CP}).

\subsection{Example 3}
On the upper half-space, define $d\gamma=x_n^k\,dx$, $x=(x_1,x_2,...,x_n) \in \rn$, where $x_n>0$, $k>0$. Then $d\gamma $ is a good measure in $\rn$. Indeed, conditions (1) and (4) follow from the validity of the following weighted isoperimetric inequality, proved in \cite{MadSalsa} when $n=2$, and in \cite{BCM3,CabreRos} whenever $n>2$ (see also \cite{ABCMP1,ABCMP2,ABCMP3}).
\begin{theorem}
 Let $\Omega $ be any bounded Lipschitz subset  of $\rn$
  having finite weighted perimeter and let $\Omega ^{\sharp
}$ be the upper half-ball centered at some $x_0$, belonging to the upper half-space $\{x\in \rn\,:\, x_n>0\}$, that is
\begin{equation*}\label{halfball}
\Omega^\sharp=B_r(x_0)\cap \left\{ x=(x_1,x_2,...,x_n) \in\R^n:\>x_n>0\right\} ,  
\end{equation*}
where $r>0$ is taken so that $|\Omega^\sharp|_\gamma=|\Omega|_\gamma.$ 
 Then
\begin{equation}
P_{\gamma }(\Omega )\geq P_{\gamma }(\Omega ^{\sharp}).
\label{dis_isoppower}
\end{equation}
Moreover equality holds in (\ref{dis_isoppower}) if and only if $\Omega
=\Omega ^{\sharp},$ modulo a translation.
\end{theorem}
\noindent A P\'olya-Sz\'eg\H{o} type principle also holds true with respect to  measure $d\gamma$ (see \cite{T}), and therefore condition (3) in Definition \ref{dens} holds true. Indeed, 
according to the definition of weighted rearrangement given in the previous section, 
we can define the rearrangement $u^\sharp$ with respect to
measure $d\gamma$ of any measurable function $u\, :\Omega \rightarrow \R$ as the rearrangement of $u$
whose super-level sets are upper half-balls having the same weighted measure as the corresponding super-level sets of $u$. 
By definition $u^{\sharp }$ is actually a radial function, depending on the only variable  $r=|x-x_0|$, and it is a decreasing function with respect to $r$.
\begin{theorem} 
 Let  $u\in H^1(\Omega,d\gamma )$. Then $u^\sharp \in H_0^1(\Omega^\sharp,d\gamma)$ and 
\begin{equation*}
\int_{\Omega}|\nabla u|^{2}d\gamma \geq \int_{\Omega^\sharp}|\nabla u^{\sharp}|^{2}d\gamma.  \label{Polya_Szego_power}
\end{equation*}
 Moreover equality holds if and only if $u=u^{\sharp},$ modulo a translation.
\end{theorem}

\noindent Finally condition (2)  of Definition \ref{dens} is also satisfied (see, for example,   \cite{BCM3}).

\subsection{Counter-examples}
The  measures $d\delta_1(x)=e^{|x|^2}dx$ or $d\delta_2(x)=x_n^k\,e^{|x|^2}dx$, with  $x_n>0$ and $k>0$, are not good measures, since they do not satisfy condition (4) in Definition \ref{dens}.
Indeed, the isoperimetric sets with respect to the measure $d\delta_1$ are concentric balls around the origin, while isoperimetric sets with respect to the measure $d\delta_2$ are concentric half-balls around the origin contained in the  upper half-space $\{x\in \rn\, :\, x_n>0   \}$ (see \cite{CRS,C,FM,RCBM}).  

\bigskip

%%%%%%%%%%%%%%%%%%%%%%%%%%%%

\section{Basic properties of weighted twisted eigenvalues and eigenfunctions}

In this section we prove some properties of weighted twisted eigenvalues with respect to a good measure. We  start with some classical ones, whose proofs use standard arguments. We repeat them here for the sake of completeness.  
\begin{proposition}
	\label{exist}
Assume that $d\gamma$ is a good measure. Then problem \eqref{mineigen} admits a minimizer $ u\in H^1_0(\Omega, d\gamma)$.
\end{proposition}
\begin{proof}  
Let $\{\phi_m\}_m$ be a minimizing sequence and assume, without loss of generality, that 
\begin{equation}\label{norm1}
\displaystyle\int_{ \Omega } |\phi_m|^2  \, d\gamma=1\,, \qquad \displaystyle\int_{ \Omega } \phi_m  \, d\gamma=0.
\end{equation}
Hence $\{\phi_m\}_m$ is bounded in $H_0^1(\Omega,d\gamma)$, 
and, by the compactness of the embedding $H_0^1(\Omega,d\gamma) \hookrightarrow L^2(\Omega,d\gamma)$,  
 there exists a function $u\in  H^1_0(\Omega, d\gamma)$ such that
$$\phi_m\rightharpoonup u \quad \hbox{weakly in } H^1_0(\Omega, d\gamma),$$ 
$$\phi_m \to u \quad \hbox{strongly in } L^2(\Omega,d\gamma) \> \mbox{and a.e.}.$$
From \eqref{norm1} we immediately deduce that $\int_\Omega u^2\, d\gamma=1$ and $\int_\Omega u \, d\gamma=0$, while as a consequence of the Banach-Steinhaus theorem we obtain
$$
\liminf_{m\to + \infty} \int_\Omega  |\nabla \phi_m|^2 \, d\gamma \geq \int_\Omega  |\nabla u|^2 \, d\gamma,
$$
that is 
$$
\lambda_{1,\gamma}^T(\Omega)=\int_\Omega  |\nabla u|^2 \, d\gamma.
$$
\end{proof} 
\noindent We continue by explicitly observing that $\lambda_{1,\gamma}^T(\Omega)$ is monotone decreasing with respect to the 
inclusion of sets (extending test functions by zero gives the inclusion of the test spaces).

\begin{proposition} Let $d\gamma$ be a good measure and let $m>0$. If $\tilde \Omega$ is an optimal set for the problem 
\begin{equation}\label{vv1}
\min\left\{ \lambda_{1,\gamma}^T(\Omega): \>\Omega \subset \R^n, \Omega \mbox{ open},  |\Omega|_\gamma=m\right\},
\end{equation}
and $\tilde \Omega$ is not connected, then the following properties hold true:
\begin{itemize}
\item[$(i)$] an eigenfunction $\tilde u$ corresponding to $\lambda_{1,\gamma}^T(\tilde\Omega)$ cannot vanish on any connected component of $\tilde\Omega$;
\item[$(ii)$] if $\tilde\Omega$ has two connected components, then $\lambda_{1,\gamma}^T(\tilde\Omega)$  is simple. 
\end{itemize}
\end{proposition}

\begin{proof} $(i)$ Assume by contradiction that $\tilde\Omega=\tilde\Omega_1\cup\tilde\Omega_2$, with $\tilde\Omega_1, \tilde\Omega_2\ne \emptyset$, $\tilde\Omega_1\cap \tilde\Omega_2=\emptyset$, and $\tilde u=0$ on $\tilde\Omega_1$. Let $t>1$ so that $|t \,\tilde\Omega_2|_\gamma=|\tilde\Omega|_\gamma$. Hence
$$
\lambda_{1,\gamma}^T(t\,\tilde\Omega_2)
<\lambda_{1,\gamma}^T(\tilde\Omega_2)=\lambda_{1,\gamma}^T(\tilde\Omega),
$$
which is absurd, being $\tilde\Omega$ an optimal set for problem \eqref{vv1}.

$(ii)$ The claim is a consequence of $(i)$. Indeed, the two connected components of $\tilde \Omega$ coincide with the nodal domains of any eigenfunction. Therefore two non zero eigenfunctions cannot be orthogonal. 
\end{proof}

\begin{remark}
We stress that, in general, $\lambda_{1,\gamma}^T(\Omega)$ can be multiple. For example, if $\Omega$ is the unitary disk in $\R^2$ and $d\gamma=dx$, $\lambda_{1,\gamma}^T(\Omega)$ is double (see \cite{BB}). This is a clue to balls not being optimal sets.
\end{remark}

The following result concerns  the boundary behavior of an eigenfunction corresponding to  the first weighted twisted eigenvalue of an optimal set $\tilde\Omega$ for problem \eqref{vv1}. Its proof is based on the fact that  $\lambda_{1,\gamma}^T(\tilde \Omega)$ is  simple and on standard  domain derivative arguments (see, for example, \cite{HP}).
\begin{proposition}\label{lemma2}
Assume that the optimal set  $\tilde\Omega$ for problem \eqref{vv1} is smooth. Denoted by $\tilde u$ an eigenfunction corresponding to $\lambda_{1,\gamma}^T(\tilde\Omega)$, there exists a positive constant $\mu$ such that
 $$
 |\nabla \tilde u|^2=\mu\mbox{  on } \partial \tilde\Omega.
 $$
 \end{proposition}

 \begin{proof}
Let us write the formula for the derivative of the weighted twisted eigenvalue. We will obtain it through a formal computation. For its rigorous proof, that would involve an implicit function Theorem together with Fredholm alternative, we remind the interested reader to \cite[pg. 257--262]{HP}. 

We assume we are given an open set $\Omega$ and a family of applications $\Phi(t)$ satisfying
$$
\Phi: t \in [0,T) \to W^{1,\infty}(\rn) \> \mbox{differentiable at } 0 \>\mbox{with } \Phi(0)=I, \Phi'(0)=V,
$$
where $V$ is a smooth vector field from $\R^n$ to $\R^n$. We denote by $u_t\in H_0^1(\Omega_t,d\gamma)$ the eigenfunction corresponding to $\lambda_{1,\gamma}^T(\Omega_t)$ normalized so that
\begin{equation}\label{norma1}
\int_{\Omega_t} u_t \, d\gamma=0 \quad \mbox{and} \quad \int_{\Omega_t} u_t^2 \, d\gamma=1.
\end{equation}
Let us denote by $u'$ and $(\lambda_{1,\gamma}^T)'$ the derivatives at 0 of $t \to u_t$ and $ t \to \lambda_{1,\gamma}^T(\Omega_t)$, respectively. The differentiation of \eqref{prob} and \eqref{norma1} leads to
\begin{equation}\label{der}
\left\{
\begin{array}{ll}
-\diver\left(\gamma(x) \nabla u'\right)=(\lambda_{1,\gamma}^T)'\gamma(x) u+\lambda_{1,\gamma}^Tu'-c' \quad & \mbox{in } \Omega
\\ \\
u'=-\dfrac{\partial u}{\partial {\bf n}}\, V\cdot{\bf n} & \mbox{on } \partial \Omega, \\ \\  \dint_\Omega u'\, d\gamma=0, \quad \dint_{\Omega} u\, u' d\gamma=0,
\end{array}
\right.
\end{equation}
where $\bf{n}$ stands for the outer unit normal to $\partial \Omega$. If we multiply the equation in \eqref{der} by $u$ and integrate by parts, we get
\begin{equation}\label{derivative}
(\lambda_{1,\gamma}^T)'(\Omega)=-\int_{\partial\Omega} \left(\frac{\partial u}{\partial {\bf n}}\right)^2 V\cdot {\bf n}\, \gamma(x)\, d\mathcal{H}^{n-1}.
\end{equation}
Since $\tilde\Omega$ is a smooth optimal set,  it is a critical point for the function $ t \to \lambda_{1,\gamma}^T(\Omega_t)$, under the volume constraint. Therefore
$$
\int_{\partial\Omega} \left[\left (\frac{\partial u}{\partial {\bf n}}\right)^2 -\mu\right]\,V\cdot {\bf n}\, \gamma(x)\, d\mathcal{H}^{n-1}=0,
$$ 
for any smooth vector field $V$, being $\mu$ the Lagrange multiplier.
This immediately infers the claim.
 \end{proof}
 
 \begin{remark}
 Note that we could directly differentiate $\lambda_{1,\gamma}^T(\Omega_t)=\int_{\Omega_t}|\nabla u_t|^2\, d\gamma$, with the constraint $\int_{\Omega_t}u_t\, d\gamma=0$ and  the normalization $ \int_{\Omega_t} u_t^2 \, d\gamma=1$.
\end{remark}

Now let us consider the first eigenvalue $\lambda_{1,\gamma}^D(\Omega)$ of the following weighted Dirichlet problem
\begin{equation}\label{probD}
\left\{\begin{array}{ll}
-\diver\left(\gamma(x)\nabla u\right)= \lambda \gamma(x) \,u \quad &\mbox{in } \Omega
\\ 
u=0 & \mbox{on } \partial\Omega,
\end{array}
\right.
\end{equation}
that is
\begin{equation*}\displaystyle
\lambda_{1,\gamma}^D(\Omega)=\min
 \left \{
 \frac{\displaystyle\int_\Omega|\nabla \psi|^2\, d\gamma}{\displaystyle\int_\Omega \psi^2\, d\gamma}\,\, : \,\,\, 
 \psi\in H^1_0(\Omega, \gamma)\, , \, \, \psi\ne 0
 \right \}.
\end{equation*}
The following results can be easly proved by adapting the proofs in \cite{FH} (see also \cite{H}).
\begin{proposition}
	\label{prop:botheigenv}
A positive number $\lambda$ is an eigenvalue for both the weighted twisted and the weighted Dirichlet problems, \eqref{prob} and \eqref{probD} respectively, if and only if there exists an associated eigenfunction $u$ for the Dirichlet problem such that
$$
 \int_\Omega u\, d\gamma=0.
$$	
\end{proposition}
\noindent We observe that, if $u$ is a weighted Dirichlet eigenfunction for a set $\Omega$, then it is also an eigenfunction for any of the nodal domains that it divides $\Omega$ into, with the same eigenvalue. For the weighted twisted problem there is, of course, no analogue of this result. It is, however, possible to relate the first weighted twisted eigenvalue to the first weighted Dirichlet eigenvalue of its nodal domains.
\begin{proposition}
	\label{prop:nodalset}
Let  $\lambda_{1,\gamma}^T(\Omega)$ be  the first eigenvalue of the weighted twisted problem \eqref{prob} and denote by $u$ an associated eigenfunction. Then  $u$ has precisely two nodal domains.
 Moreover, denoted by 
\begin{equation}\label{omega+}
\Omega_{+}=\{ x\in \Omega \, :\, u(x)>0\},
\qquad
\Omega_{-}=\{ x\in \Omega\, :\, u(x)<0\}
\end{equation} 
its nodal domains, we have
\begin{equation*}\label{ineqDgamma}
\min\{\lambda_{1,\gamma}^D(\Omega_{-}), \lambda_{1,\gamma}^D(\Omega_{+}) \}
\le 
\lambda_{1,\gamma}^T(\Omega)
\le
\max\{\lambda_{1,\gamma}^D(\Omega_{-}), \lambda_{1,\gamma}^D(\Omega_{+}) \}.
\end{equation*}
\end{proposition}

\begin{proof}  By a variation of the Courant's nodal domain theorem applied to the twisted problem, any eigenfunction associated with the first eigenvalue has precisely two nodal domains (see \cite{BB}). Then the proof follows as in \cite{FH}.
\end{proof}

\noindent With analogous arguments as in \cite{FH} (see also \cite{H}) we can prove the following 
\begin{corollary}\label{cor:2.3} 
With the same notation used in Proposition \ref{prop:nodalset}, if we denote by $\lambda_{2,\gamma}^D(\Omega)$ the second weighted Dirichlet eigenvalue of $\Omega$, we have
$$\lambda_{1,\gamma}^T(\Omega)= \lambda_{1,\gamma}^D(\Omega_{+})$$ 
if, and only if, 
$$\lambda_{1,\gamma}^D(\Omega_{+})= \lambda_{1,\gamma}^D(\Omega_{-})=\lambda_{2,\gamma}^D(\Omega).$$ 
\end{corollary}

\noindent After all the above considerations, we can easily deduce that
$$
\lambda_{1,\gamma}^D(\Omega)<\lambda_{1,\gamma}^T(\Omega)\le \lambda_{2,\gamma}^D(\Omega).
$$

\begin{remark}
Actually, as in \cite{BB}, it can be shown that Dirichlet and twisted eigenvalues are intertwined, that is
$$
\lambda_{1,\gamma}^D(\Omega)<\lambda_{1,\gamma}^T(\Omega)\le \lambda_{2,\gamma}^D(\Omega)\le \lambda_{2,\gamma}^T(\Omega)\le \lambda_{3,\gamma}^D(\Omega)\le ...
$$
\end{remark}

%%%%%%%%%%%%%%%%%%%%%

\section{A general result for good measures}
The main result of this section is given by Theorem \ref{lemma1}, stating that  the first weighted twisted eigenvalue $\lambda_{1,\gamma}^T(\Omega)$ decreases if we pass from a generic open set $\Omega$ to the union of two disjoint isoperimetric sets, having the same weighted measure as the nodal sets of an eigenfunction corresponding to $\lambda_{1,\gamma}^T(\Omega)$. This implies that the optimal set for problem \eqref{vv1} has to be sought among unions of two disjoint isoperimetric sets, having the same measure as the initial set.

In order to prove  Theorem \ref{lemma1} below, we denote by $u$ an eigenfunction corresponding to $\lambda_{1,\gamma}^T(\Omega)$ and we consider the nodal domains defined in \eqref{omega+}.
Thus, the following identity holds true
\begin{equation*}\label{omega+u}
\lambda_{1,\gamma}^T(\Omega)=
\frac
{\displaystyle \int_{\Omega_{+}}|\nabla u|^2\, d\gamma+\int_{\Omega_{-}}|\nabla u|^2\, d\gamma }
{\displaystyle\int_{\Omega_{+}} u^2 d\gamma  + \int_{\Omega_{-}} u^2 d\gamma }.
\end{equation*}

\begin{theorem}\label{lemma1}
Let us denote by $\omega_-$ and $\omega_+$  
two disjoint isoperimetric sets (for the weighted perimeter with respect to the weighted measure) such that $|\omega_+|_\gamma=|\Omega_{+}|_\gamma$ and $|\omega_-|_\gamma= |\Omega_{-}|_\gamma$.
Then 
\begin{equation}\label{lemma}
\lambda_{1,\gamma}^T(\Omega)\ge \lambda_{1,\gamma}^T(\omega_{+}\cup \omega_{-}).
\end{equation}
\end{theorem}
\begin{proof} 
Let us consider the weighted rearrangements $u_\pm^\sharp$ of $u_\pm$.
By using the P\'{o}lya-Szeg\H{o} principle  \eqref{Polya_Szego_Eh} and equality \eqref{Cavalieri}, it is easy to verify that
\begin{equation}\label{ww1}
\lambda_{1,\gamma}^T(\Omega)\ge
\frac
{\displaystyle\int_{\omega_{+}} |\nabla u^\sharp_{+}|^2\, d\gamma + \int_{\omega_{-}} |\nabla u^\sharp_{-}|^2\, d\gamma 
}
{\displaystyle\int_{\omega_{+}} | u^\sharp_{+}|^2\, d\gamma + \int_{\omega_{-}} | u^\sharp_{-}|^2\, d\gamma 
}
\end{equation}
and
\begin{equation}\label{ww2}
\int_{\omega_{+}} u^\sharp_{+}\, d\gamma - \int_{\omega_{-}}  u^\sharp_{-}\, d\gamma= 
\int_{\Omega_+}  u_+\, d\gamma-\int_{\Omega_-}  u_-\, d\gamma=\int_\Omega u\,d\gamma=0\, .
\end{equation}
From \eqref{ww1} and \eqref{ww2} we immediately deduce  
$$
\lambda_{1,\gamma}^T(\Omega)\ge\lambda_{\gamma}^*,
$$
with
\begin{equation}\label{lstar}
\lambda_{\gamma}^*=
\inf
\left\{
 \frac{\displaystyle\int_{\omega_{+}}|\nabla f|^2\, d\gamma+ \int_{\omega_{+}}|\nabla g|^2\, d\gamma
 }
 {\displaystyle\int_{\omega_{+}} f^2\, d\gamma + \int_{\omega_{+}} g^2\, d\gamma}\,\ :\,\,
 f\in H^1_0(\omega_{+}, d\gamma), \,
g\in H^1_0(\omega_{-}, d\gamma),\, 
 \int_{\omega_{+}} f\, d\gamma= \int_{\omega_{-}} g\, d\gamma
 \right\}\,.
\end{equation} 
We stress that, by using classical methods of the Calculus of Variations, we can  prove that the infimum in the definition of $\lambda_\gamma^\ast$ is attained for a couple of functions that
we denote by  $(f_{+}, f_{-})$. The Euler-Lagrange equation satisfied by $f_\pm$, taking into account the condition $\int_{\omega_{-}} g\, d\gamma- \int_{\omega_{+}} f\, d\gamma= 0$, can be written as follows
\begin{equation}\label{33}
\begin{array}{ll}
\dint_{\omega_{+}} &\nabla f_{+} \cdot \nabla \phi \, d\gamma +
\dint_{\omega_{-}} \nabla f_{-} \cdot \nabla \psi \, d\gamma -
\lambda_\gamma^\ast\left ( 
\dint_{\omega_{+}}  f_{+}   \phi \, d\gamma +
\dint_{\omega_{-}}  f_{-}   \phi \, d\gamma 
   \right)=
\\ \\
&=\mu_0\left ( 
\dint_{\omega_{-}}    \psi \, d\gamma  - \dint_{\omega_{+}}     \phi \, d\gamma 
\right), \quad \phi \in H_0^1(\omega_+,d\gamma), \psi \in H_0^1(\omega_-,d\gamma),
\end{array}
\end{equation}
where $\mu_0$ is a Lagrange multiplier and, for the homogeneity of the problem, we have assumed
$$
\int_{\omega_{+}}  f_{+}^2 \, d\gamma +
\int_{\omega_{-}}  f_{-}^2 \, d\gamma =1\,.
$$
Let us firstly choose $\psi=0$, and then $\phi=0$, in \eqref{33}. We thus find that $f_\pm$ respectively  solve
\begin{equation} 
	\label{eq:Pfp}
	\begin{cases}
	-\diver\Big( \gamma(x) \nabla f_{+} \Big) =
	\lambda_{\gamma}^*\, \gamma(x) \, f_{+}-\mu_0 \, \gamma(x)&\mbox{in } \omega_{+}\\
	f_+=0 &\mbox{on } \partial \omega_{+}\,,
\end{cases}
\end{equation}
and
\begin{equation}
	\label{eq:Pfm}
	\begin{cases}
	-\diver\Big( \gamma(x) \nabla f_{-} \Big) =
	\lambda_{\gamma}^*\,\gamma(x)\,
	f_{-} +\mu_0 \,\gamma(x) &\mbox{in } \omega_{-}\\
	f_-=0 &\mbox{on } \partial \omega_{-}\,.
\end{cases}
\end{equation}
Let us integrate equation \eqref{eq:Pfp} on $\omega_+$ and equation \eqref{eq:Pfm} on $\omega_-$. By subtracting the obtained integral relations, we have
$$
- \int_{\omega_{+}} \diver\left(\gamma(x)\nabla f_+\right)\, dx
+\int_{\omega_{-}} \diver\left(\gamma(x)\nabla f_-\right)\,dx
=-\mu_0 \left(|\omega_+|_\gamma+|\omega_-|_\gamma\right)=-\mu_0|\Omega|_\gamma.
$$
We deduce
$$
\mu_0 = \frac 1 {|\Omega|_\gamma }\int_{\omega_{+}\cup \omega_{-}} \diver\left(\gamma(x) \nabla w\right)\, dx,
$$
where
$$
w= 
\begin{cases}
f_{+} &\mbox{in } \omega_{+}\\
-f_{-} &\mbox{in } \omega_{-}.
\end{cases}
$$
This implies that $w$ is a solution to the following eigenvalue problem
\begin{equation*}
	\begin{cases}
	-\diver\left(\gamma(x)\nabla w\right)=
	\lambda_{\gamma}^*\,\gamma(x)\,
	w-\gamma(x)\left(\dfrac{1}{|\Omega|_\gamma}\displaystyle \int_{\omega_{+}\cup\, \omega_{-}} \diver\left(\gamma(x)\nabla w\right)\,dx \right)
 &\mbox{in } \omega_{+}\cup \omega_{-}\\
	w=0 &\mbox{on }   \partial  \left( \omega_{+}\cup \omega_{-}\right),
\end{cases}
\end{equation*}
or, equivalently, that $\lambda^*_\gamma$ is a weighted twisted eigenvalue of the set $\omega_{+}\cup \omega_{-}$.Therefore
\begin{equation*}\label{lemma}
\lambda_{1,\gamma}^T(\Omega)\ge\lambda^*_\gamma \ge  \lambda_{1,\gamma}^T(\omega_{+}\cup \omega_{-})\,
\end{equation*}
and the claim follows.
\end{proof}

\begin{corollary}\label{existence}
Assume that the isoperimetric sets for the weighted measures are a one-parameter family of domains and that
the weighted twisted eigenvalue is continuous with respect to this parameter.
Then, for every $m>0$, the shape optimization problem \eqref{vv1}
has a solution $\tilde\Omega$.
\end{corollary}

\begin{proof}
Theorem \ref{lemma1} shows that, in order to solve the shape optimization problem, it suffices to look at unions of two disjoint isoperimetric sets. Then, the existence of an optimal set $\tilde \Omega$ immediately follows  from the continuity assumption (and boundedness of the eigenvalues).
\end{proof}

%%%%%%%%%%%%%%%%%%%%%%%%%%%%%%%%%%%%%%

\section{Proof of Theorem \ref{teo:main}: the Gaussian measure}

In the present section we assume $\gamma(x)=\pi^{-n/2} e^{-|x|^2}$. As pointed out in Section \ref{sec_dens}, $d\gamma$ is a good measure and isoperimetric sets are half-spaces. 

\noindent Before proving Theorem \ref{teo:main}, we recall a few properties of the Hermite funcitons which will be useful in the sequel (see, for example, \cite[Chapters 4 and 10]{L}).
\medskip

\subsection{Hermite functions}\label{hermite}

Let us consider the differential equation
\begin{equation}\label{1}
y''-2\,t y'+2\,\nu y=0.
\end{equation}
If $\nu=n \in \N_0$, it is satisfied by the Hermite polynomial of degree $n$, i.e. 
$$
H_n(t)=(-1)^n e^{t^2}\frac{d^n(e^{-t^2})}{dt^n}.
$$
For example,
$$
H_0(t)=1, \quad H_1(t)=2t, \quad H_2(t)=4t^2-2, \quad H_3(t)=8t^3-12t,...
$$
Therefore, when $\nu$ is an arbitrary parameter, it is natural to call Hermite functions the solutions to \eqref{1}.  The Hermite functions can be expressed in terms of the confluent hypergeometric function $\Phi(\alpha, \beta; t)$ as follows
$$
y=A\Phi\left(-\frac \nu 2,\frac 1 2; t^2\right)+Bt\Phi\left(\frac{1-\nu}{2},\frac 3 2 ; t^2 \right).
$$
In particular, choosing the constants $A$ and $B$ to be
$$
A=\frac{2^\nu\, \Gamma\left(\frac 1 2 \right)}{\Gamma\left(\frac{1-\nu}{2}\right)}, \quad B=\frac{2^\nu \,\Gamma\left(-\frac 1 2 \right)}{\Gamma\left(-\frac \nu 2\right)},
$$
we arrive at the solution 
$$
y=H_\nu(t)=\frac{2^\nu\, \Gamma\left(\frac 1 2 \right)}{\Gamma\left(\frac{1-\nu}{2}\right)}\Phi\left(-\frac \nu 2,\frac 1 2; t^2\right)+\frac{2^\nu\, \Gamma\left(-\frac 1 2 \right)}{\Gamma\left(-\frac \nu 2\right)}t\Phi\left(\frac{1-\nu}{2},\frac 3 2 ; t^2 \right),
$$
called the Hermite function of degree $\nu$.

If $\nu=n\in\N_0$, it can be easily verified that $H_\nu(t)$ reduces to the Hermite polynomial of degree $n$.

If $\nu \notin \N_0$, the general solution to equation \eqref{1} can be written in terms of Hermite functions. Indeed, since equation \eqref{1} does not change if we replace $t$ by $-t$, both functions $H_\nu(t)$ and $H_\nu(-t)$ verify it. Moreover, their wronskian is equal to
\begin{equation*}\label{wronsk}
W(H_\nu(t), H_\nu(-t))=\frac{2^{\nu+1}\sqrt{\pi}}{\Gamma(-\nu)}e^{t^2},
\end{equation*}
and it never vanishes. Hence, if $\nu \notin \N_0$, the solutions $H_\nu(t)$ and $H_\nu(-t)$ are linearly independent and the general solution to \eqref{1} can be written as follows
\begin{equation}\label{sol}
y(t)=M H_\nu(t)+N H_\nu(-t), \quad M, N \in \R.
\end{equation}
Clearly, the general solution to a non-homogeneous equation of the type
$$
y''-2\,t y'+2\,\nu y=c, \qquad c \in \R,
$$
has the following form
$$
y(t)=M H_\nu(t)+N H_\nu(-t)+\frac{c}{2\nu}, \quad M, N \in \R.
$$

Unfortunately, if $\nu=n\in\N_0$, $W(H_\nu(t), H_\nu(-t))=0$, hence $H_\nu(t)$ and $H_\nu(-t)$ are linearly dependent. In fact,
$$
H_n(-t)=(-1)^n H_n(t).
$$
Therefore, \eqref{sol} is no longer the general solution to \eqref{1}. For arbitrary natural $\nu$, it can be proven that the general solution to \eqref{1} can be written in either of the following equivalent forms:
$$
y(t)=M H_\nu(t)+N e^{t^2} H_{-\nu-1}(it)=P H_\nu(t)+Q e^{t^2} H_{-\nu-1}(-it), \quad M, N, P, Q \in \R.
$$
The behavior of the Hermite functions for large positive $t$ and arbitrary $\nu$  is described by the following identity
\begin{equation}\label{bai}
H_\nu(\pm t)=\pm (2|t|)^\nu\left[\sum_{k=0}^N \frac{(-1)^k(-\nu)_{2k}}{k!}(2t)^{-2k}+O\left(|t|^{-2N-2}\right)\right],
\end{equation}
where $(-\nu)_{2k}=(-\nu)(-\nu-1)\cdots (-\nu-2k+1)$.
Moreover the following recurrance relation holds true
\begin{equation}\label{rec}
H_\nu'(t)=2\nu H_{\nu-1}(t).
\end{equation}
It is also known that $H_\nu(t)$ has no positive zeros if $0<\nu<1$, while it has $m$ positive zeros if $2 m -1<\nu<2m+1$,  $m\ge 1$. In the particular case of Hermite polynomials, that is $\nu =n \in \N$, $H_n(t)$ has $n$ zeros in the interval $\left[-2\sqrt{\frac{n+1}{2}},2\sqrt{\frac{n+1}{2}}\right]$ (see, for example, \cite{Dean}).

In \cite{IV} the following Turan's type inequality is proved to hold true
\begin{equation}\label{turan}
H_\nu(t)^2-H_{\nu-1}(t)H_{\nu+1}(t)>0, \quad t>h_{1,\nu},
\end{equation}
where $h_{1,\nu}$ is the leftmost zero of $H_\nu(t)$.

\medskip
\subsection{Proof of Theorem \ref{teo:main} }
\label{sec:ineq}

By Theorem \ref{lemma1} inequality \eqref{lemma} holds true with $\omega_+, \omega_-$  half-spaces. Therefore  in order to conclude the proof of Theorem \ref{teo:main}, it remains to prove that the union of two symmetric half-spaces, hence having the same Gaussian measure, gives the lowest possible value of $\lambda_{1,\gamma}^T$ in the class
$$
\mathcal{H}=\left\{H_\ell\cup H_r \subset \R^n: \> |H_\ell|_\gamma\le |H_r|_\gamma, \> |H_\ell\cup H_r|_\gamma=|\Omega|_\gamma\right\},
$$
where
$$
H_\ell=\left\{x\in \R^n: \> x_1<-\ell\right\}, \qquad H_r=\left\{x\in \R^n: \> x_1>r\right\}.
$$

Let $\tilde \Omega$ be the optimal set whose existence is guaranteed by Corollary \ref{existence}. Denote
$$
\tilde \Omega=H_L\cup H_R\,,
$$
for some $L,R>0$. We prove that $\tilde \Omega$ is given whenever $R=L$. Assume by contradiction that $R<L$. We claim that there exists an eigenfunction $u$ corresponding to $\lambda_{1,\gamma}^T(\tilde \Omega)$, that does not satisfy Proposition \ref{lemma2}.

\noindent We proceed by dividing the proof into three steps.
\medskip

\noindent \emph{Step 1. The eigenfunction $u$.} Let us consider the following function, defined in $\tilde \Omega$,
\begin{equation}\label{u}
u=\left\{
\begin{array}{ll}
u_L=A\left(H_\nu(-x_1)-H_\nu(L)\right), \quad & x_1<-L
\\ \\
u_R=B\left(H_\nu(R)-H_\nu(x_1)\right), & x_1>R,
\end{array}
\right.
\end{equation}
where $2\nu =\lambda_{1,\gamma}^T(H_L\cup H_R)$. This function $u$ is an eigenfunction of the problem \eqref{prob} with $\Omega=\tilde\Omega$. Indeed,  $u_L$ and $u_R$ are rearranged functions with respect to the Gaussian measure, that is they depend only on  the variable $x_1$, according to Theorem \ref{lemma1}. Moreover, $u_L$ and $u_R$ satisfy the following boundary value problems, respectively
\begin{equation}\label{eqL}
\left\{
\begin{array}{ll}
y''-2\,x_1\,y'+\lambda_{1,\gamma}^T(H_L\cup H_R)y=c \quad \mbox{if } x<-L
\\ \\
y \in H_0^1\left((-\infty,-L);d\gamma_1\right),
\end{array}
\right.
\end{equation}
and 
\begin{equation}\label{eqR}
\left\{
\begin{array}{ll}
y''-2\,x_1\,y'+\lambda_{1,\gamma}^T(H_L\cup H_R)y=c \quad \mbox{if } x>R
\\ \\
y \in H_0^1\left((R,+\infty);d\gamma_1\right),
\end{array}
\right.
\end{equation}
where $d\gamma_1=e^{-x_1^2}/\sqrt{\pi}\,dx_1$ and $c$ is the constant, unknown a priori,  which corresponds to the term $\frac{1}{|\Omega|_\gamma}\int_\Omega \diver\left(\gamma(x)\nabla u\right)\, dx$.
Furthermore, since a variant of Courant's Theorem on nodal domains holds true in this context, $u_L$ is positive in $H_L$ and $u_R$ is negative in $H_R$. 

\medskip

\noindent \emph{Step 2. The constant $c$.} 
In this step we give the explicit expression of the constant $c$, which  
is the same in both problems \eqref{eqL} and \eqref{eqR}, getting that 
\begin{equation}\label{eq1}
A H_\nu(L)+B H_\nu(R)=0.
\end{equation}
Indeed, if we compute $u'_L$ and $u''_L$ by differentiating the expression contained in \eqref{u},  and then we replace them in the equation in \eqref{eqL}, by using  \eqref{1}, we  deduce
\begin{equation}\label{c1}
c=-2\nu A H_\nu(L).
\end{equation}
Analogously, we compute  $u'_R$ and $u''_R$, we substitute them in the equation in  \eqref{eqR} and, by using  \eqref{1}, we gain
\begin{equation}\label{c2}
c=2\nu B H_\nu(R).
\end{equation}
Combining \eqref{c1} and \eqref{c2} yields \eqref{eq1}.
\medskip

\noindent \emph{Step 3. Conclusion}. 
We prove that 
$
u'(R)^2-u'(L)^2>0,
$
which gives  a contradiction with the optimality condition stated in Proposition \ref{lemma2}.

\noindent 
Combining \eqref{u} and \eqref{eq1} yields
\begin{equation}\label{f}
u'(R)^2-u'(L)^2=B^2H_\nu'(R)^2-A^2 H_\nu'(L)^2=A^2 H_\nu(L)^2\left[\left(\frac{H_\nu'(R)}{H_\nu(R)}\right)^2-\left(\frac{H_\nu'(L)}{H_\nu(L)}\right)^2\right].
\end{equation}
We now use the recurrence relation \eqref{rec} in \eqref{f} and we obtain 
$$
u'(R)^2-u'(L)^2=4\nu^2 A^2 H_\nu(L)^2\left[\left(\frac{H_{\nu-1}(R)}{H_\nu(R)}\right)^2-\left(\frac{H_{\nu-1}(L)}{H_\nu(L)}\right)^2\right].
$$
Denote
$$
\psi_\nu(x_1)=\frac{H_{\nu-1}(x_1)}{H_\nu(x_1)}, \qquad x_1>R.
$$
The function $\psi_\nu(x_1)$ is well-defined since, as recalled in Subsection \ref{hermite}, the Hermite function $H_\nu$ has a finite number of zeros and since $u_R$ is negative in the half-space $H_R$, $R$ is greater than the largest zero of $H_\nu$. 

\noindent Moreover $\psi_\nu(x_1)$ is strictly decreasing in $[R,+\infty[$. Indeed, using \eqref{rec} we can compute 
$$
\psi_\nu'(x_1)=\frac{H'_{\nu-1}(x_1)H_\nu(x_1)-H_{\nu-1}(x_1)H_\nu'(x_1)}{H_\nu(x_1)^2}=\frac{2(\nu-1)H_{\nu-2}(x_1)H_\nu(x_1)-2\nu H_{\nu-1}(x_1)^2}{H_\nu(x_1)^2}.
$$
If $\nu-1<0$, we immediately get that $\psi_\nu'(x_1)<0$. If $\nu-1>0$, we can use Turan's inequality \eqref{turan} and we get to the same conclusion.
On the other hand, since by \eqref{bai} $H_\nu(x_1) \sim x_1^\nu $ as $x_1 \to +\infty$, $\psi_\nu(x_1) \to 0$ as $x_1 \to +\infty$, so that $\psi_\nu(x_1)$ is strictly positive. Therefore the following inequality holds true
$$
u'(R)^2-u'(L)^2=4\nu^2\alpha^2 H_\nu(L)^2\left(\psi_\nu(R)^2-\psi_\nu(L)^2\right)>0\,,$$
and we get a contradiction. Hence, the optimal set $\tilde \Omega$ coincides with the union of the two symmetric half-spaces having the same Gaussian measure, and \eqref{mainineq} holds true.

%%%%%%%%%%%%%%%%%%%%%%%%%%%%%%%%%%%%%%%%%%%

\section{Proof of Theorem \ref{teo:main}: the power-like measure}

In the present section we assume $\gamma(x)=x_n^k$, $k>0$, $x_n>0$. As pointed out in Section \ref{sec_dens}, $d\gamma$ is a good measure and half-balls contained in the upper half-space  $\{x\in \rn\, :\, x_n>0 \}$,  centered at some point $x_0 \in \{x\in \rn\,:\, x_n=0\}$, are isoperimetric sets. 

\noindent Before proving Theorem \ref{teo:main}, we recall a few properties of the Bessel functions which will be useful in the sequel (see, for example, \cite{OBL}, \cite{Wat}).
\medskip

\subsection{Bessel functions}\label{bessel}

Let us consider the differential equation
\begin{equation}\label{2}
r^2y''+ry'+(r^2-\alpha^2)y=0, \quad \alpha \in \R, r \in \R.
\end{equation}
It is well-known that it is solved by  Bessel functions $J_\alpha$, $Y_\alpha$ of order $\alpha$ of the first and second kind.  $J_\alpha, Y_\alpha$ are linearly independent for any value of $\alpha$. Moreover, the following relation holds true (for non integer $\alpha$)
\begin{equation}\label{jy}
Y_\alpha(r)=\frac{J_\alpha(r)\cos(\alpha \pi)-J_{-\alpha}(r)}{\sin(\alpha \pi)},
\end{equation}
and where the right-hand side is replaced by its limiting value whenever $\alpha$ is an integer. Moreover, $J_\alpha$ satisfies the following fundamental recurrence relation
\begin{equation}\label{consrecbess}
rJ_\alpha'(r)-\alpha J_{\alpha}(r)=-rJ_{\alpha+1}(r),  \qquad r \in \R.
\end{equation}
If we denote by $j_{\alpha,h}, j_{\alpha,h}'$ the zeros of $J_\alpha, J_\alpha'$, respectively, then
$$
\alpha\le j'_{\alpha,1}<j_{\alpha,1}<j'_{\alpha,2}<....
$$
and
$$
 j_{\alpha,1}<j_{\alpha+1,1}<j_{\alpha,2}<....
$$
Moreover in \cite{OBL} the following identities can be found
\begin{equation}
\label{jnu}
\begin{aligned}
J_\alpha(r)&=\frac{\left(\frac 1 2 r\right)^\alpha}{\Gamma(\alpha+1)} \prod_{h=1}^\infty \left(1-\frac{r^2}{j_{\alpha,h}^2}\right), \quad \alpha \ge 0
\\
J_\alpha'(r)&=\frac{\left(\frac 1 2 r\right)^{\alpha-1}}{2\Gamma(\alpha)} \prod_{h=1}^\infty \left(1-\frac{r^2}{(j_{\alpha,h}')^2}\right), \quad \alpha > 0.
\end{aligned}
\end{equation}

\medskip

\subsection{Proof of Theorem \ref{teo:main}}
\label{sec:ineq2}

By Theorem \ref{lemma1} inequality \eqref{lemma} holds true with $\omega_+, \omega_-$  half-balls contained in the  upper half-space $\{x\in \rn\, :\, x_n>0   \}$ and  centered on the hyperplane $\{x_n=0\}$. Therefore,  in order to conclude the proof of Theorem \ref{teo:main}, it remains to prove that the union of two disjoint half-balls contained in the  upper half-space $\{x\in \rn\, :\, x_n>0   \}$, centered on the hyperplane $\{x_n=0\}$, sharing the same weighted measure, gives the lowest possible value of $\lambda_{1,\gamma}^T$ in the class
$$
\mathcal{B}=\left\{B_L\cup B_R \subset \R^n: \> |B_L|_\gamma\ge|B_R|_\gamma, \> |B_L\cup B_R|_\gamma=|\Omega|_\gamma\right\},
$$
with
$$
B_L=\left\{x \in \R^n, x_n>0: \> ||x-O_L||<L\right\}, \quad B_R=\left\{x \in \R^n, x_n>0: \> ||x-O_R||<R\right\},
$$
being $||O_L-O_R||>L+R$.
We proceed by dividing the proof into three steps.
\medskip

\noindent \emph{Step 1. The eigenfunctions.} 
Let $u$ be an eigenfunction corresponding to $\lambda_{1,\gamma}^T(B_L\cup B_R)$, which does not vanish on any of the two half-balls. We denote
$$
u=\left\{
\begin{array}{ll}
u_L & \mbox{in } B_L
\\
u_R & \mbox{in } B_R.
\end{array}
\right.
$$
As in the case of the Gaussian measure, without loss of generality, we can assume $u_L>0$ in $B_L$, $u_R<0$ in $B_R$, and $u_L, u_R$ radially symmetric. Then $u_L$ and $u_R$ solve the following boundary value problems respectively
\begin{equation}\label{eqb1}
\left\{
\begin{array}{ll}
ry''+(n+k-1)y'+\lambda^2r y=cr \quad & \mbox{if } r \in (0,L) 
\\ \\
y \in H_0^1(B_L,d\gamma)
\end{array}
\right.
\end{equation}
and
\begin{equation}\label{eqb2}
\left\{
\begin{array}{ll}
ry''+(n+k-1)y'+ \lambda^2ry=cr \quad & \mbox{if } r \in (0,R) 
\\ \\
y \in H_0^1(B_R,d\gamma),
\end{array}
\right.
\end{equation}
where $y(r)=y(|x-x_0|),$ with $r=|x-x_0|$, $c$ is the constant, unknown a priori,  which corresponds to the term $\frac{1}{|\Omega|_\gamma}\int_\Omega \diver\left(\gamma(x)\nabla u\right)\, dx$ and $\lambda^2=\lambda_{1,\gamma}^T(B_L\cup B_R)$. 

\noindent
Taking into account properties of Bessel functions, recalled in Subsection \ref{bessel},  the general solution to the equation which appears in both problems \eqref{eqb1} and \eqref{eqb2} is given by
$$
y(r)=r^\alpha\left(c_1 J_\alpha(\lambda r)+c_2 Y_\alpha (\lambda r)\right)+\frac{c}{\lambda^2}, \quad c_1,c_2\in \R,
$$
where $\alpha=1-\frac{n+k}{2}<0$.  
Moreover, if   $\alpha$ is not an integer, \eqref{jy} allows us to write $y(r)$ as follows
$$
y(r)=r^\alpha\left(c_1 J_\alpha(\lambda r)+c_2 J_{-\alpha} (\lambda r)\right)+\frac{c}{\lambda^2}, \quad c_1,c_2\in \R.
$$
Since the solutions to both problems \eqref{eqb1}, \eqref{eqb2} must have finite energy, that is  
 $y \in H_0^1(B_L,d\gamma)$ and $y \in H_0^1(B_L,d\gamma)$ respectively, having \eqref{jnu} in mind, we must choose $c_1=0$
 (indeed when $r\to 0$ the gradient of $J_\alpha$ times $x_n^k$ behaves like $1/r$ and is not integrable) and we can write
\begin{equation}\label{u1}
u=\left\{
\begin{array}{ll}
u_L=A \left(|x-O_L|^\alpha J_{-\alpha}(\lambda |x-O_L|)-L^\alpha J_{-\alpha}(\lambda L)\right)\,, \quad & \mbox{if } x \in B_L
\\ \\
u_R=B\left(R^\alpha J_{-\alpha}(\lambda R)- |x-O_R|^\alpha J_{-\alpha}(\lambda |x-O_R|)\right)\,, \quad & \mbox{if } x \in B_R.
\end{array}
\right.
\end{equation}
We can argue exactly in the same way when $\alpha$ is an integer, obtaining the same formula \eqref{u1}.
Recalling the graph of the Bessel function $J_{-\alpha}(x)$, since $u_L$, $u_R$ have constant sign in $B_L$, $B_R$, respectively, and they are monotone with respect to $r$, we immediately infer that $\lambda L, \lambda R   \in [0,j_{-\alpha,1}'].$

\medskip

\noindent \emph{Step 2. The constant $c$.}  The constant $c$ in both problems \eqref{eqb1}, \eqref{eqb2} is the same, so that
\begin{equation}\label{eq2}
AL^\alpha J_{-\alpha}(\lambda L)+BR^\alpha J_{-\alpha}(\lambda R)=0.
\end{equation}
Indeed, if we compute $u'_L$ and $u''_L$ by differentiating the explicit expression contained in \eqref{u1},  and we replace them in the equation in \eqref{eqb1}, by using  \eqref{2}, we  deduce
\begin{equation}\label{cb1}
c=-A \lambda^2 L^\alpha J_{-\alpha}(\lambda L)\,.
\end{equation}
Analogously, we compute  $u'_R$ and $u''_R$, we substitute them in the equation in  \eqref{eqb2} and, by using  \eqref{2}, we  obtain
\begin{equation}\label{cb2}
c=B \lambda^2 R^\alpha J_{-\alpha}(\lambda R),
\end{equation}
By combining \eqref{cb1} and \eqref{cb2}, we deduce \eqref{eq2}.
\medskip

\noindent \emph{Step 3. The optimal set $\tilde \Omega$.} In this step we prove that the optimal set $\tilde \Omega$ is given whenever $R=L$. The existence of $\tilde \Omega$ is guaranteed by Corollary \ref{existence}. 

 \noindent Assume by contradiction that $R<L$. 
Combining \eqref{u1} and \eqref{eq2}  yields
\begin{eqnarray*}
u'(L)^2-u'(R)^2
&=&A^2L^{2\alpha -2} \left(\alpha J_{-\alpha}(\lambda L)+L\lambda J_{-\alpha}'(\lambda L)\right)^2-B^2 R^{2\alpha -2}\left(\alpha J_{-\alpha}(\lambda R)+R\lambda J_{-\alpha}'(\lambda R)\right)^2
\\
& =& A^2 L^{2\alpha} J_{-\alpha}(\lambda L)^2 \lambda^2\left[
\left(\frac {\alpha J_{-\alpha}(\lambda L)+\lambda L J_{-\alpha}'(\lambda L)}{\lambda L J_{-\alpha}(\lambda L)}\right)^2
-
\left(\frac {\alpha J_{-\alpha}(\lambda R)+\lambda R J_{-\alpha}'(\lambda R)}{\lambda R J_{-\alpha}(\lambda R)}\right)^2
\right].
\\
\end{eqnarray*}
Let us define
$$
\phi_\alpha(s)=\frac {\alpha J_{-\alpha}(s)+s J_{-\alpha}'(s)}{s J_{-\alpha}(s)}\,,\qquad s\in [0,  j'_{-\alpha,1}]$$
or equivalently, via \eqref{consrecbess}, 
$$
\phi_\alpha(s)=-\frac {J_{-\alpha+1}(s)}{ J_{-\alpha}(s)}\,,\qquad s\in [0,  j'_{-\alpha,1}].
$$
From \cite[Section 4]{landau} we deduce that $\phi_\alpha(s)$ is strictly decreasing in $[0,j'_{-\alpha,1}]$ and negative.
This provides that 
$$
u'(L)^2-u'(R)^2<0,
$$
which gives a contradiction with Proposition \ref{lemma2}. Hence, the optimal set $\tilde \Omega$ is given by the union of two disjoint half-balls, having the same weighted measure, and \eqref{mainineq} follows.

\begin{remark}
We explicitly observe that, by taking $k=0$ in the previous arguments, meaning $d\gamma=dx$, and recalling that isoperimetric sets are balls, we provide here an alternative proof to the result contained in \cite{FH}.
\end{remark}

%%%%%%%%%%%%%%%%%%%%%%%%%%%%%%%%%%%%%%%%%%%

%%%%%%%%%%%%%%%%%%%%%%%%%%%%%%
\vskip 1cm

(Barbara Brandolini), Dipartimento di Matematica e Informatica, Universit\`a degli Studi di Palermo, via Archirafi 34, 90123 Palermo, Italy; email: \texttt{barbara.brandolini@unipa.it}
\smallskip

(Antoine Henrot), Universit\'e de Lorraine, CNRS, Institut Elie Cartan de Lorraine, F-54000 Nancy, France,
email: \texttt{antoine.henrot@univ-lorraine.fr}

\smallskip
(Anna Mercaldo) Dipartimento di Matematica e Applicazioni ``R. Caccioppoli'', Universit\`a degli Studi di Napoli Federico II, Complesso Monte S. Angelo - via Cintia, 80126 Napoli, Italy; email: \texttt{anna.mercaldo@unina.it}

\smallskip
(Maria Rosaria Posteraro) Dipartimento di Matematica e Applicazioni ``R. Caccioppoli'', Universit\`a degli Studi di Napoli Federico II, Complesso Monte S. Angelo - via Cintia, 80126 Napoli, Italy; email: \texttt{posterar@unina.it}

\end{document}